\documentclass[12pt,a4paper]{article}
\usepackage{graphicx}
\usepackage{amsmath}
\usepackage{amsthm}
\usepackage{amsfonts}
\usepackage{amssymb}

\usepackage{fullpage}

 \newtheorem{theorem}{Theorem}
 \newtheorem{proposition}{Proposition}
 \newtheorem{lemma}{Lemma}
 \newtheorem{corollary}{Corollary}
 \theoremstyle{definition}
 
 \newtheorem{example}{Example}

\newcommand\N{\mathbb{N}}
\newcommand\Z{\mathbb{Z}}

\newcommand\R{\mathbb{R}}

\newcommand\X{\mathbb{X}}

\begin{document}

\title{Decaying positive global solutions of second order difference equations with mean curvature operator
	\\[-2mm]}

\author{Zuzana Do\v sl\'a\footnote{Corresponding author. Email: dosla@math.muni.cz}, Serena Matucci and Pavel \v Reh\'ak
}
\date{}
\maketitle \vspace*{-1.1cm}

\begin{center}
\begin{minipage}{10cm}
\begin{center}
\noindent
\textit{\small
Department of Mathematics and Statistics \\
Masaryk University, Brno, Czech Republic\\
\texttt{\small dosla@math.muni.cz}\\
Department of Mathematics and Informatics\\
University of Florence, Florence, Italy\\
\texttt{\small serena.matucci@unifi.it} \\	
Inst. of Math., FME, Brno University of Technology\\
Technick\'a 2,
CZ--61669 Brno,
Czech Republic
\texttt{\small rehak.pavel@fme.vutbr.cz}
}
\end{center}
\end{minipage}
\end{center}

\medskip
\begin{center}
to appear on \emph{Electronic Journal of Qualitative Theory of Differential Equations}
\end{center}
\medskip

\begin{abstract}
A boundary value problem on an unbounded domain,
associated to difference equations with the Euclidean mean curvature
operator is considered. The existence of solutions which are positive on the whole domain and decaying at infinity is examined by proving new Sturm comparison theorems for linear difference equations and using a fixed point approach based on a linearization device.
The process of discretization of the boundary value problem on the unbounded domain is examined, and some discrepancies between the discrete  and the continuous case are pointed out, too.
\end{abstract}
\begin{center}
\textit{Dedicated to the 75th birthday of Professor Jeff Webb}
\end{center}



\noindent

\section{Introduction} 

In this paper we study the boundary value problem (BVP) on the half-line for difference equation with the Euclidean mean curvature operator
\begin{equation}
\Delta\left(a_k\frac{\Delta x_k}{\sqrt{1+(\Delta x_k)^2}}\right)
+b_kF(x_{k+1})=0,
\label{E}
\end{equation}
subject to the conditions
\begin{equation}
x_m=c,\quad x_k>0,\quad \Delta x_k\le 0,
 \quad \displaystyle \lim_{k\to\infty}x_k=0,
\label{BVP}
\end{equation}
where $m \in \Z^{+}=\N \cup \{0\}$, $k\in\Z_m:=\{k \in \Z: \, k \geq m\}$ and $c\in(0,\infty)$.

Throughout the paper the following conditions are assumed:
\begin{itemize}
\item[(H$_{1}$)] The sequence $a$ satisfies $a_k>0$ for $k\in\Z_m$ and
\begin{equation*}
	\label{sum-a}
	\sum_{j=m}^{\infty}\frac{1}{a_j}<\infty.
\end{equation*}
\item[(H$_{2}$)] The sequence $b$ satisfies $b_k\ge 0$  for $k\in\Z_m$ and	
\begin{equation*}
\label{sum-a-b}
\sum_{j=m}^{\infty}b_j\sum_{i=j}^{\infty}\frac{1}{a_i}<\infty.
\end{equation*}
\item[(H$_{3}$)] The function $F$ is continuous on $\R$, $F(u)u>0$ for $ u\ne 0$, and
\begin{equation}
	\label{F}
\lim_{u\to 0+}\frac{F(u)}{u}<\infty.	
\end{equation}
\end{itemize}

When modeling real life phenomena, boundary value problems for second order differential equations play important role. The BVP \eqref{E}-\eqref{BVP} originates from the
discretisation process for searching radial solutions, which are globally positive and decaying,
for PDE with Euclidean  mean curvature operator.
By globally positive solutions we mean solutions which are positive  on the whole domain $\mathbb{Z}_{m}$. The Euclidean mean curvature operator arises in the study of some fluid mechanics problems, in particular
capillarity-type phenomena for compressible and incompressible fluids.

Recently, discrete BVPs, associated to equation \eqref{E}, have been widely studied, both in bounded and unbounded domains, see, e.g., \cite{abgo} and references therein.
Many of these papers can be seen as a finite dimensional variant of results established in the continuous case. For instance, we refer to \cite{Mawhin07,Mawhin08,MawhinTaylor,Mawhin2013} for BVPs involving mean curvature operators in Euclidean and Minkowski spaces, both in the continuous  and in the  discrete case.  Other results in this direction are in \cite{BC2008,BC2009}, in which  the multiplicity of solutions of certain BVPs involving the $p$-Laplacian is examined.  Finally, in \cite{JDEA2016,DS} for second order equations with $p$-Laplacian the existence of  globally positive decaying Kneser solutions, that is solutions $x$ such that $x_n>0$, $\Delta x_n<0$ for $n\geq 1$ and $\lim_{n\to\infty}x_n=0$, is examined.

Several approaches have been used in literature for treating the above problems. Especially, we refer to variational methods \cite{Mawhin2012}, the critical point theory \cite{BC2009} and fixed point theorems on cones \cite{Webb06,Webb17}.

Here, we extend to second order difference equations with Euclidean mean curvature some results on  globally positive decaying Kneser solutions stated in \cite{JDEA2016} for equations with $p$-Laplacian and $b_n<0$.

This paper  is motivated also by \cite{Trieste}, in which BVPs for differential equation with the Euclidean mean curvature operator on the half-line $[1,\infty)$ have been studied subjected to the boundary conditions $x(1)=1$ and $\lim_{t\to\infty}x(t)=0$. The study in \cite{Trieste} is accomplished by using a linearization device and some properties of principal solutions of certain disconjugate second-order linear differential equations. Here,  we consider the discrete setting of  the problem studied in  \cite{Trieste}. However, the discrete analogue presented here requires different technique. This is caused by a different behavior of decaying solutions as well as by peculiarities of the discrete setting which lead to a modified fixed point approach. Jointly with this, we prove new Sturm comparison theorems and new properties of recessive solutions for linear difference equations.
Our existence result is based on a fixed point theorem for operators defined in a Fr\' echet space by a Schauder's linearization device. This method is originated in \cite{CFM}, later extended to the discrete case in \cite{MMR2}, and recently developed in \cite{Fixed}. This tool does not require the explicit form of the fixed point operator $T$ and simplifies the check of the topological properties of $T$ on the unbounded domain, since these properties become an immediate consequence of \textit{a-priori} bounds for an associated linear equation. These bounds are obtained in an implicit form by means of the concepts of recessive solutions for second order linear equations. The main properties and results which are  needed in our arguments, are presented in Sections \ref{S2} and \ref{S3}. In Section \ref{S4} the solvability of the BVP (\ref{E})-(\ref{BVP}) is given, by assuming some implicit conditions on sequences $a$ and $b$.
Several effective criteria are given, too. These criteria are obtained  by considering
suitable linear equations which can be viewed as Sturm majorants of the auxiliary linearized equation.
In Section \ref{S5} we compare our results with those stated in the continuous case in \cite{Trieste}.
Throughout the paper we emphasis some discrepancies, which arise between the continuous case and the discrete one.

\section{Discrete versus continuous decay}\label{S2}

Several properties in the discrete setting have no continuous analogue. For instance, for a positive sequence $x$ we always have
\[
\frac{\Delta x_k}{x_k}=\frac{x_{k+1}}{x_k}-1>-1.
\]
In the continuous case, obviously, this does not occur in general, and the decay can be completely different. For example, if $x(t)=e^{-2t}$ then $x'(t)/x(t)=-2$ for all $t$. Further, the ratio $x'/x$ can be also unbounded from below, as the function  $x(t)=\mathrm{e}^{-\mathrm{e}^t}$ shows.

Another interesting observation is the following. If two positive continuous functions $x,y$ satisfy the inequality
\[
\frac{x'(t)}{x(t)}\le M\frac{y'(t)}{y(t)}, \ \ t\geq t_0,
\]
then there exists $K>0$ such that $x(t)\le Ky^M(t)$ for $t\geq t_0$.
This is not true in the discrete case, as the following example illustrates.

\begin{example}
Consider the sequences $x,y$ given by
\[
x_k=\frac{1}{2^{2^k}}, \quad y_k=\frac{1}{2^{2^{k+2}}}.
\]
Then
\[
\frac{x_{k+1}}{x_k}=\frac{1}{2^{2^k}}, \quad \frac{y_{k+1}}{y_k}=\frac{1}{2^{2^{k+2}}},
\]
and
\[
\frac{\Delta x_{k}}{x_k}= \frac{1}{2^{2^k}}-1\leq \frac{1}{2}-1=- \frac{1}{2}\leq \frac{1}{2}\left( \frac{1}{2^{2^{k+2}}}-1\right)= \frac{1}{2}\frac{\Delta y_{k}}{y_k}
\]
 On the other hand, the inequality $x_k\le K y_k^{1/2}$ is false for every value of $K>0$. Indeed,
 \[
 \frac{x_k}{\sqrt{y_k}}= \frac{2^{2^{k+1}}}{2^{2^{k}}}=2^{2^{k}}
\]
which is clearly unbounded.
\end{example}

The situation in the discrete case is described in the following two lemmas.
\begin{lemma}
	\label{L:A}
Let $x,y$ be positive sequences on $\Z_m$ such that $M\in(0,1)$ exists, satisfying
\begin{equation}\label{xy}
\frac{\Delta x_k}{x_k}\le M\frac{\Delta y_k}{y_k}	
\end{equation}
for $k\in\Z_m$. Then $1+M\Delta y_k/y_k>0$ for $k\in\Z_m$, and
$$
x_k\le x_m\prod_{j=m}^{k-1}\left(1+M\frac{\Delta y_j}{y_j} \right).
$$
\end{lemma}

\begin{proof}
First of all note that, from $M \in (0,1)$ and the positivity of $y$, we have
\[
1+M\frac{\Delta y_k}{y_k}=1+M\frac{y_{k+1}}{y_k}-M>0, \quad k \in \Z_m.
\]
From \eqref{xy} we get
\[
\frac{x_{k+1}}{x_k}\le 1+M\frac{\Delta y_k}{y_k},
\]
and taking the product from $m$ to $k-1$, $k>m$,  we obtain
$$
\frac{x_k}{x_m}=\frac{x_{m+1}}{x_m}\,\frac{x_{m+2}}{x_{m+1}}\cdots\frac{x_k}{x_{k-1}}\le \prod_{j=m}^{k-1}\left(1+M\frac{\Delta y_j}{y_j} \right).
$$
\end{proof}

From the classical theory of infinite products (see for instance \cite{Knopp}) the infinite product $P=\prod_{k=m}^\infty (1+q_k)$
of real numbers is said to \emph{converge} if there is $N\in\Z_m$ such that $1 +q_k \neq 0$ for $k\geq N$ and
\[
P_n=\prod_{k=N}^n (1 +q_k)
\] has a \emph{finite and nonzero} limit as $n \to \infty$.

In case $-1<q_k\leq0$,  $\{P_n\}$  is a positive nonincreasing sequence, thus $P$ being \emph{divergent} (not converging to a nonzero number) means that
\begin{equation}\label{limP}
\lim _{n\to\infty} \prod_{k=N}^n (1 +q_k) =0.
\end{equation}
Moreover, the convergence of  $P$ is equivalent to the convergence of the series
$\sum_{k=N}^\infty \ln (1+q_k)$ and this is equivalent to the convergence of the series
$\sum_{k=N}^\infty  q_k$. Indeed,
if $\sum_{k=m}^{\infty} q_k$ is convergent, then
$\lim_{k\to\infty} q_k=0$ and hence,
\[
\lim_{k\to\infty} \frac{\ln (1+q_k)}{q_k}=1,
\]
i.e., $\ln(1+q_k)\sim q_k$ as $k\to\infty$.
Since summing preserves asymptotic equivalence, we get that $\sum_{k=m}^{\infty}\ln(1+q_k)$ converges. Similarly, we obtain the opposite direction.

Therefore, in case $-1<q_k\leq0$,  \eqref{limP} holds
if and only if $\sum_{k=N}^\infty q_k$ diverges to $-\infty$.

The following holds.

\begin{lemma} \label{L:B}
Let $y$ be a positive nonincreasing sequence on $\Z_m$ such that
$\lim_{k\to\infty}y_k=0$. Then, for any $M\in(0,1)$,
$$
\lim_{k\to\infty}\prod_{j=m}^{k}\left(1+M\frac{\Delta y_j}{y_j} \right)=0.
$$	
\end{lemma}

\begin{proof}
From the theory of infinite products  it is sufficient to show that
\begin{equation}
\sum_{j=m}^{\infty}\frac{\Delta y_j}{y_j}=-\infty.
\label{div}
\end{equation}
We distinguish two cases:\\
1) there exists $N>0$ such that $y_{k+1}/y_k\ge N$ for $k\in\Z_m$;\\
2) $\inf_{k\in\Z_m}y_{k+1}/y_k=0$.

As for the former case, from the Lagrange mean value theorem, we have
$$
-\Delta \ln y_k=-\frac{\Delta y_k}{\xi_k}\le-\frac{\Delta y_k}{y_{k+1}}=
-\frac{\Delta y_k}{y_k}\cdot\frac{y_k}{y_{k+1}}
\le -\frac{\Delta y_k}{Ny_k},
$$
where $\xi_k$ is such that $y_{k+1}\le\xi_k\le y_k$ for $k\in\Z_m$. Summing the above inequality from $m$ to $n-1$, $n>m$,   we get
$$
\ln y_m-\ln y_n \le -\frac{1}{N}\sum_{j=m}^{n-1}\frac{\Delta y_j}{y_j}.
$$
Since  $\lim_{n\to\infty}y_n=0$, letting $n \to \infty$ we get (\ref{div}).

Next we deal with the case  $\inf_{k\in\Z_m}y_{k+1}/y_k=0$. This is equivalent to
	\[\liminf_{k\to\infty}\frac{\Delta y_{k}}{y_k}=\liminf_{k\to\infty} \frac{y_{k+1}}{y_k}-1=-1,\]
	which implies (\ref{div}),
since
$\sum_{j=m}^{k}\Delta y_j/y_j$ is negative nonincreasing.

\end{proof}

\section{A Sturm-type comparison theorem for linear equations}\label{S3}

The main idea of our approach is based on an application
of a fixed point theorem and on global monotonicity properties of recessive solutions
of linear equations.  To this goal,  in this section we prove a new Sturm-type comparison theorem for linear difference equations.

Consider the linear equation
\begin{equation} \label{Lmin}
\Delta(r_k\Delta y_k)+p_k y_{k+1}=0,
\end{equation}
where $p_k\ge 0$ and $r_k>0$ on $\Z_m$.
We say that a solution $y$ of equation \eqref{Lmin} has a generalized zero in $n$
if  either $y_n= 0$ or $y_{n-1}y_{n}< 0$, see e.g. \cite{Agarwal, abe:dis}.
A (nontrivial) solution $y$ of \eqref{Lmin} is said to be \textit{nonoscillatory}
if $y_ky_{k+1}>0$ for all large $k$. Equation \eqref{Lmin} is said to be \textit{nonoscillatory}
if all its nontrivial solutions are nonoscillatory.
It is well known that, by the Sturm type separation theorem, the nonoscillation of \eqref{Lmin}
 is equivalent to the existence of a nonoscillatory
solution  see e.g. \cite[Theorem~1.4.4]{abgo}, \cite{abe:dis}.

If \eqref{Lmin} is nonoscillatory, then there exists a nontrivial solution $u$, uniquely determined up to a constant factor,
such that
$$
\lim_{k\to\infty}\frac{u_k}{y_k}=0,
$$
where $y$ denotes an arbitrary nontrivial solution of \eqref{Lmin}, linearly
independent of $u$. Solution $u$ is called \textit{recessive solution} and $y$ a \textit{dominant solution}, see e.g. \cite{ahlbrandt}.
Recessive solutions can be characterized in the following ways (both these properties are proved in \cite{ahlbrandt}):\\
(i)
A solution $u$ of \eqref{Lmin} is recessive if and only if
\begin{equation*}\label{int_char}
\sum_{j=m}^{\infty}\frac{1}{r_ju_ju_{j+1}}=\infty.
\end{equation*}
(ii)
For a recessive solution $u$ of \eqref{Lmin} and any linearly independent solution $y$ (i.e. dominant solution) of \eqref{Lmin}, one has
\begin{equation} \label{minimal}
\frac{\Delta u_k}{u_k}< \frac{\Delta y_k}{y_k} \qquad\text{ eventually.}
\end{equation}

Along with equation \eqref{Lmin} consider
the equation
\begin{equation} \label{Lmaj}
\Delta(R_k\Delta x_k)+P_k x_{k+1}=0	
\end{equation}
where $P_k\ge p_k\ge 0$ and $0<R_k\le r_k$ on $\Z_m$; equation \eqref{Lmaj} is  said to be a \textit{Sturm majorant} of \eqref{Lmin}.

From \cite[Lemma~1.7.2]{abgo}, it follows that if \eqref{Lmaj} is nonoscillatory, then \eqref{Lmin} is nonoscillatory as well. In this section we always assume that \eqref{Lmaj} is nonoscillatory.

\medskip
The following two propositions are slight modifications of results in \cite{dos-reh}. They are preparatory to the main comparison result.

\begin{proposition}[{\cite[Lemma 2]{dos-reh}}] \label{P:0}
Let $x$ be a positive solution of \eqref{Lmaj} on $\Z_m$ and $y$ be a solution
of \eqref{Lmin} such that $y_m>0$ and $r_m\Delta y_m/y_m\ge R_m\Delta x_m/x_m$. Then
$$
y_k>0 \quad\text{ and } \quad \frac{r_k\Delta y_k}{y_k}\ge\frac{R_k\Delta x_k}{x_k}, \ \text{ for  } k\in\Z_m.
$$
Moreover, if $y,\bar y$ are solutions of \eqref{Lmin} such that $y_k>0$, $k\in\Z_{m}$, and $\bar y_m>0$,
$\Delta \bar y_m/\bar y_m>\Delta y_m/y_m$, then
$$
\bar y_k>0 \quad\text{ and } \quad \frac{\Delta\bar y_k}{\bar y_k}>\frac{\Delta y_k}{y_k}, \ \text{ for } k\in\Z_m.
$$
\end{proposition}

\begin{proposition}[{\cite[Theorem 3]{dos-reh}}] \label{Prop1}
If a recessive solution $v$ of \eqref{Lmin} has a generalized zero in $N\in \Z_m$
and has no generalized zero in $(N,\infty)$, then any solution of \eqref{Lmaj} has a generalized zero in
$(N-1,\infty)$.
\end{proposition}

The following lemma is an improved version of \cite[Theorem 1]{dos-reh}.
\begin{lemma}  \label{L:0.5}
Let $u, v$ be  recessive solutions of \eqref{Lmin} and \eqref{Lmaj}, respectively,  satisfying $u_k>0, v_k>0$ for $k\in\Z_m$.
Then
\begin{equation}\label{hallo_Zuzana}
\frac{r_k\Delta u_k}{u_k}\le \frac{R_k\Delta v_k}{v_k} \quad \text{ for } k\in\Z_m.
\end{equation}
\end{lemma}

\begin{proof}
By contradiction, assume that there exists $N\in\Z_m$ such that $r_N\Delta u_N/u_N> R_N\Delta v_N/v_N$. Let $y$ be a solution of
\eqref{Lmin} satisfying $y_N >0$ and $r_N\Delta y_N/y_N=R_N\Delta v_N/v_N$. Then $r_N\Delta y_N/y_N<r_N\Delta u_N/u_N$, (which implies that $y$ is linearly independent with $u$) and from Proposition~\ref{P:0} we get
$y_k>0$, $\Delta y_k/y_k<\Delta u_k/u_k$ for $k\in\Z_N$, which contradicts \eqref{minimal}.
\end{proof}

\begin{lemma}
	\label{L:1}
Let $x$ be a positive solution of \eqref{Lmaj} on $\Z_m$. Then there exists a recessive solution
$u$ of \eqref{Lmin}, which is positive on $\Z_m$.	
\end{lemma}

\begin{proof}
Let $u$ be a recessive solution of \eqref{Lmin}, whose existence is guaranteed by nonoscillation of majorant equation \eqref{Lmaj}. By contradiction, assume that there exists $N\in\Z_m$ such that
\[
u_N \neq 0, \quad u_N u_{N+1}\le 0.
\]
Then $u$ cannot have a generalized zero in $(N+1,\infty)$. Indeed, if $u$ has a generalized zero in $M\in\Z_{N+2}$, then
by the Sturm comparison theorem on a finite interval (see e.g., \cite[Theorem~1.4.3]{abgo}, \cite[Theorem~1.2]{abe:dis}), every solution of \eqref{Lmaj} has a generalized zero in  $(N,M]$, which is a contradiction with the positivity of $x$.
Applying now Proposition~\ref{Prop1}, we get that any solution of \eqref{Lmaj} has a generalized zero in $(N,\infty)$ which again contradicts  the positivity of $x$ on $\Z_m$.
\end{proof}

The next theorem is, in fact, the main statement of this section and it plays an important role in the proof of Theorem~\ref{Tmain}.

\begin{theorem} \label{Trec}
Let $x$ be a positive solution of \eqref{Lmaj} on $\Z_m$.
Then there is a  recessive solution $u$ of \eqref{Lmin}, which is positive on $\Z_m$ and  satisfies
\begin{equation} \label{ruuRxx}
\frac{r_k\Delta u_k}{u_k}\le\frac{R_k\Delta x_k}{x_k},\ \ k\in\Z_m.
\end{equation}
In addition, if $x$ is decreasing (nonincreasing) on $\Z_m$, then $u$ is decreasing (nonincreasing) on $\Z_m$.	
\end{theorem}

\begin{proof}
Let $x$ be a positive solution of \eqref{Lmaj} on $\Z_m$. From Lemma~\ref{L:1}, there exist a recessive solution $u$ of \eqref{Lmin} and a recessive solution $v$ of \eqref{Lmaj}, which are both positive on $\Z_m$. 	
We claim that
\begin{equation} \label{hallo_Serena}
\frac{\Delta v_k}{v_k}\le\frac{\Delta x_k}{x_k} \quad \text{for } k \in\Z_m.
\end{equation}
Indeed, suppose by contradiction that there is $N\in\Z_m$ such that $\Delta x_N/x_N<\Delta v_N/v_N$. Then, in view of Proposition~\ref{P:0},
$\Delta x_k/x_k<\Delta v_k/v_k$ for $k\in\Z_N$, which contradicts \eqref{minimal}.
 Combining \eqref{hallo_Serena} and \eqref{hallo_Zuzana},
we obtain \eqref{ruuRxx}. The last assertion of the statement is an immediate consequence of \eqref{ruuRxx}.
\end{proof}

Taking $p=P$ and $r=R$ in Theorem~\ref{Trec}, we get the following corollary.

\begin{corollary} \label{corol}
If \eqref{Lmaj} has a positive decreasing (nonincreasing) solution on  $\Z_m$, then there exists a recessive solution of \eqref{Lmaj} which is positive decreasing (nonincreasing)  on  $\Z_m$.
\end{corollary}

We close this section by the following characterization of the asymptotic behavior of recessive solutions which will be used later.

\begin{lemma}
\label{L:4}
Let
$$
\sum_{j=m}^\infty \frac{1}{r_j}<\infty\ \ \text{and}\ \
\sum_{j=m}^{\infty}p_j\sum_{i=j+1}^{\infty} \frac{1}{r_i}<\infty.
$$	
Then \eqref{Lmin} is nonoscillatory. Moreover, for every $d\neq 0$, \eqref{Lmin} has an eventually positive,
nonincreasing recessive solution $u$, tending to zero and satisfying
$$
\lim_{k\to\infty} \frac{u_k}{\sum_{j=k}^{\infty}r_j^{-1}}=d.
$$
\end{lemma}

\begin{proof} It follows from \cite[Lemma 2.1 and Corollary 3.6]{cdm-ade}. More precisely,
the result \cite[Lemma~2.1]{cdm-ade} guarantees $\lim_{k\to\infty}r_k\Delta u_k=-d<0$. Now, from the discrete L'Hospital rule, we get
$$
\lim_{k\to\infty}\frac{u_k}{\sum_{j=k}^{\infty}r_j^{-1}}=
\lim_{k\to\infty}\frac{\Delta u_k}{-r_k^{-1}}=d.
$$
\end{proof}

\section{Main result: solvability of BVP}\label{S4}

Our main result is the following.

\begin{theorem}\label{Tmain}
Let (H$_{i}$), i=1,2,3, be satisfied and
\begin{equation}\label{Lc}
L_c=\sup_{u\in(0,c]}\frac{F(u)}{u}.
\end{equation}
 If the linear difference equation
\begin{equation}
	\label{L2}
\Delta\left(\frac{a_k}{\sqrt{1+c^2}}\,\Delta z_k \right)
+L_c b_k z_{k+1}=0,	
\end{equation}
 has a positive decreasing solution on $\Z_m$, then BVP \eqref{E}-\eqref{BVP} has at least one solution.
\end{theorem}

Effective criteria, ensuring the existence of a positive decreasing solution of \eqref{L2}, are given at the end of this section.
\vskip2mm

From this theorem and its proof we get the following.

\begin{corollary}\label{c:1}
Let (H$_{i}$), i=1,2,3, be satisfied. If \eqref{L2} has a  positive decreasing solution on $\Z_m$ for $c=c_0>0$, then \eqref{E}-\eqref{BVP} has at least one solution for every  $c\in (0, c_0]$.
\end{corollary}
\vskip4mm

To prove Theorem \ref{Tmain}, we use a fixed point approach,
based on the Schauder-Tychonoff theorem on the Fr\'echet space
\[
\mathbb X=\{u: \Z_m \to \R\}
\]
of all sequences defined on $\Z_m$, endowed with the topology
of pointwise convergence on $\mathbb{Z}_{m}$.
The use of the Fr\'{e}chet space $\mathbb{X}$, instead of a
suitable Banach space, is advantageous especially for the compactness test. Even if this is true also in the
continuous case, in the discrete case the situation is even more simple, since any bounded set in $\mathbb{X}$ is
relatively compact from the discrete
Arzel\`{a}-Ascoli theorem. We recall that a set $\Omega \subset \mathbb{X}$ is bounded if the sequences in $\Omega$
are equibounded on every compact subset of $\Z_{m}$. The compactness test is therefore very
simple just owing to the topology of $\mathbb{X}$, while in discrete Banach spaces can
require some checks which are not always immediate.

Notice that, if
$\Omega\subset$ $\mathbb{X}$ is bounded, then $\Omega^{\Delta}=\{\Delta
u,\,u\in\Omega\}$ is bounded, too. This is a significant discrepancy between the continuous and the discrete case; such a property can simplify the solvability of discrete boundary value problems associated to equations of order two or higher with respect to the continuous counterpart because \textit{a-priori} bounds for the first difference
\[
\Delta x_n=x_{n+1}-x_n
\]
are a direct consequences of \textit{a-priori} bounds for $x_n$, and similarly for higher order differences.

\vskip2mm
In \cite[Theorem 2.1]{MMR2}, the authors proved an existence result for BVPs associated to functional difference equations in Fr\'echet spaces (see also \cite[Corollary 2.6]{MMR2}, \cite[Theorem 4]{Fixed} and remarks therein). That result is a  discrete counterpart of an existence result stated in \cite[Theorem 1.3]{CFM} for the continuous case, and reduces the problem to that of finding good \textit{a-priori} bounds for the unknown of a auxiliary linearized equation.

The function
\begin{equation*}
\Phi(v)=\frac{v}{\sqrt{1+v^{2}}}
\end{equation*}
can be decomposed as
\begin{equation*}\label{JJ}
\Phi(v)=vJ(v),
\end{equation*}
 where $J$ is a continuous function on $\R$ such that
\,${\lim_{v\to 0}J(v)=1}$. This suggests the form of an auxiliary linearized equation.
Using  the same arguments as in the proof of \cite[Theorem 2.1]{MMR2}, with minor changes, we have the following.

\begin{theorem}
\label{T:FPT} Consider the (functional) BVP
\begin{equation}
\begin{cases}
\Delta(a_{n} \Delta x_{n}J(\Delta x_{n}))=g(n,x), & n\in\mathbb{Z}_{m},\\
x\in S, &
\end{cases}
\label{DF}
\end{equation}
where $J: \R \to \R$ and $g:\mathbb{Z}_{m}\times\mathbb{X}\rightarrow\mathbb{R}$ are continuous
maps, and $S$ is a subset of $\ \mathbb{X}$. \newline Let $G:\,\mathbb{Z}_{m}\times\mathbb{X}^{2}\rightarrow\mathbb{R}$ be a continuous map such that
$G(k,q,q)=g(k,q)$ for all $(k,q)\in\mathbb{Z}_{m}\times\mathbb{X}$. If there
exists a nonempty, closed, convex and bounded set $\Omega\subset\mathbb{X}$ such that: \\
a) \ for any $u\in\Omega$ the problem
\begin{equation}
\begin{cases}
\Delta(a_{n} J(\Delta u_{n}) \Delta y_{n})=G(n,y,u), & n\in\mathbb{Z}_{m},\\
y\in S, &
\end{cases}
\label{DF1}
\end{equation}
has a unique solution $y= T(u)$;\\
b) \ $T(\Omega)\subset \Omega$;\\
c) \ $\overline{T(\Omega)} \subset S$,\\
then \eqref{DF} has at least one solution.
\end{theorem}
\begin{proof} We briefly summarize the main arguments, for reader's convenience, which are a minor modification of the ones in \cite[Theorem 2.1]{MMR2}.

Let us show that the operator $T:\Omega\to \Omega$ is continuous with relatively compact image. The relatively compactness of $T(\Omega)$ follows immediately from b), since $\Omega$ is bounded. To prove the continuity of $T$ in $\Omega$, let $\{u^j\}$ be a sequence in $\Omega$, $u^j \to u^\infty \in \Omega$, and let $v^j=T(u^j)$. Since $T(\Omega)$ is relatively compact, $\{v^j\}$ admits a subsequence (still indicated with $\{v^j\}$) which is convergent to $v^\infty$, with $v^\infty \in S$ from c). Since $J, G$ are continuous on their domains, we obtain
\[
0= \Delta(a_{n} J(\Delta u^j_{n}) \Delta v^j_{n})-G(n,v^j,u^j) \to \Delta(a_{n} J(\Delta u^\infty_{n}) \Delta v^\infty_{n})-G(n,v^\infty,u^\infty)
\]
as $j\to\infty$.
The uniqueness of the solution of \eqref{DF1} implies $v^\infty=T(u^\infty)$, and therefore $T$ is continuous. By the Schauder-Tychonoff fixed point theorem, $T$ has at least one fixed point in $\Omega$, which is clearly a solution of \eqref{DF}.
\end{proof}

\medskip

\begin{proof}[\bf Proof of Theorem \ref{Tmain}.]
Let $z$ be the recessive solution of \eqref{L2} such that $z_m=c$, $z_k>0$, $\Delta z_k\le 0$, $k\in\Z_m$; the existence of a recessive solution with these properties follows from Corollary~\ref{corol}.
Further, from Lemma~\ref{L:4}, we have $\lim_{k\to\infty}z_k=0$.

Define the set $\Omega$ by
$$
\Omega=\left\{u\in\X : 0\le u_k\le c\prod_{j=m}^{k-1}\left(1+M\frac{\Delta z_j}{z_j}\right),k\in\Z_m \right\},
$$
where $\X$ is the Fr\'echet space of all real sequences defined on $\Z_m$, endowed with the topology of pointwise convergence on $\Z_m$, and $M=1/\sqrt{1+c^2}\in(0,1)$. Clearly $\Omega$ is a closed, bounded and convex  subset of $\X$.

For any $u\in\Omega$, consider the following BVP
\begin{equation}
\label{L1}
\begin{cases}
\Delta\left(\dfrac{a_k}{\sqrt{1+(\Delta u_k)^2}} \, \Delta y_k\right)
+b_k \tilde F(u_{k+1}) y_{k+1}=0, & k \in \Z_m, \displaystyle\\[2mm]
y \in S \displaystyle
\end{cases}
\end{equation}
where
\[
\tilde F(v)=\frac{F(v)}{v} \quad \text{for } v > 0, \quad
\tilde F(0)= \lim_{v \to 0^+} \frac{F(v)}{v}
\]
is continuous on $\R^{+}$, due to assumption \eqref{F}, and
$$
S=\left\{y\in\X: y_m=c,\, y_k>0, \Delta y_k \leq 0 \text{ for }k\in\Z_m, \,\sum_{j=m}^{\infty} \frac{1}{a_j y_j y_{j+1}}=\infty \right\}.
$$
Since $0\le u_k\le c$, for every $u \in \Omega$, we have $-c\le\Delta u_k\le c$, and so $(\Delta u_k)^2\le c^2$. Therefore,
\[
\frac{1}{\sqrt{1+(\Delta u_k)^2}}\ge \frac{1}{\sqrt{1+c^2}}
\]
for every $u \in \Omega$ and $k\in\Z_m$. Further
$\tilde F(u_{k+1})\le L_c$ for $u \in \Omega$, and hence \eqref{L2} is Sturm majorant for the linear equation in \eqref{L1}. Let $\widehat{y}=\widehat{y}(u)$ be the recessive solution of the equation in \eqref{L1} such that $\widehat{y}_m=c$. Then $\widehat{y}$ is positive nonincreasing on $\Z_m$ by Theorem~\ref{Trec}, and, in view of $\widehat{y}_m=c$ and the uniqueness of recessive solutions up to the constant factor, $\widehat y $ is the unique solution of \eqref{L1}. Define the operator $\mathcal{T}:\Omega\to\X$ by $(\mathcal{T} u)_k=\widehat{y}_k$ for $u\in\Omega$.

From Theorem~\ref{Trec}, we get
$$
\frac{a_k\Delta\widehat{y}_k}{\widehat{y}_k}\le
\frac{a_k\Delta\widehat{y}_k}{\widehat{y}_k\sqrt{1+(\Delta u_k)^2}}\le
\frac{a_kM\Delta z_k}{z_k}\le 0,
$$
which implies $\Delta\widehat{y}_k/\widehat{y}_k\le M\Delta z_k/z_k$, $k\in\Z_m$. By Lemma~\ref{L:A},
\[
\widehat y_k\le c\prod_{j=m}^{k-1}\left (1+M \frac{\Delta z_j}{z_j}\right), \quad k\in\Z_m,
\]
which yields $\mathcal{T}(\Omega)\subseteq\Omega$.

Next we show that
$\overline{\mathcal{T}(\Omega})\subseteq S$. Let $\overline{y}\in\overline{\mathcal{T}(\Omega})$. Then there exists
$\{u^j\}\subset\Omega$ such that $\{\mathcal{T}u^j\}$  converges to $\overline{y}$ (in the topology of $\X$).  It is not restrictive to assume $\{u^j\}\to \bar u \in \Omega$ since $\Omega$ is compact.
Since $\mathcal{T}u^j=:\widehat{y}^j$ is the (unique) solution of \eqref{L1}, we have
$\widehat{y}_m^j=c$,  $\widehat{y}^j_k>0$ and $\Delta\widehat{y}^j_k\le 0$ on $\Z_m$ for every $j\in\N$.
Consequently, $\overline{y}_m=c$, $\overline{y}_k\geq 0$, $\Delta\overline{y}_k\le 0$
for $k\in\Z_m$. Further, since $\tilde F$ is continuous, $\overline{y}$ is a solution of the equation in \eqref{L1} for  $u=\overline{u}$. Suppose now that there is
$T\in\Z_m$ such that $\overline{y}_T=0$. Then clearly $\Delta \overline{y}_T=0$ and by the global existence and uniqueness
of the initial value problem associated to any linear equation, we get $\overline{y}\equiv 0$ on $\Z_m$, which contradicts to $\overline{y}_m=c>0$. Thus $\overline{y}_k > 0$ for all $k\in\Z_m$.

We have just to prove that $\sum_{j=m}^{\infty}(a_j\overline{y}_j\overline{y}_{j+1})^{-1}=\infty$.
In view of
Lemma~\ref{L:4}, there exists $N>0$ such that $\overline{y}_k\le N\sum_{j=k}^{\infty} a_j^{-1}$ on $\Z_m$.
Noting that
$$
\Delta\left(\frac{1}{\sum_{j=k}^\infty a_j^{-1}} \right)
=\frac{1}{a_k\sum_{j=k}^\infty a_j^{-1}\sum_{j=k+1}^{\infty}a_j^{-1}},
$$
we obtain
\begin{align*}
\sum_{j=m}^{k-1}\frac{1}{a_j\overline{y}_j\overline{y}_{j+1}}\ge
&\sum_{j=m}^{k-1}\frac{1}{N^2a_j\sum_{i=j}^\infty a_i^{-1}\sum_{i=j+1}^{\infty}a_i^{-1}}=\frac{1}{N^2}\sum_{j=m}^{k-1} \Delta \left(\frac{1}{\sum_{i=j}^\infty a_i^{-1}}\right)\\
=&\frac{1}{N^2}\left(\frac{1}{\sum_{j=k}^\infty a_j^{-1}}-\frac{1}{\sum_{j=m}^\infty a_j^{-1}}\right)
\to\infty \text{ as }k\to\infty.
\end{align*}
 Thus  $\overline{y}\in S$, i.e., $\overline{\mathcal{T}(\Omega})\subseteq S$. By applying Theorem~\ref{T:FPT}, we obtain that the problem
\begin{equation*}
\begin{cases}
\Delta\left(a_k\dfrac{\Delta x_k}{\sqrt{1+(\Delta x_k)^2}}\right)
+b_kF(x_{k+1})=0,\quad k \in \Z_m,\\	
x \in S
\end{cases}
\end{equation*}
has at least a solution $\bar x \in \Omega$.
From the definition of the set $\Omega$,
\[
\overline{x}_k\le c\prod_{j=m}^{k-1}\left(1+M\frac{\Delta z_j}{z_j}\right)
\]
and since $M\in(0,1)$ and $\lim_{k\to\infty}z_k=0$, we have
\[
\lim_{k\to\infty}\prod_{j=m}^{k-1}\left(1+M\frac{\Delta z_j}{z_j}\right)=0
\]
by Lemma~\ref{L:B}. Thus $\bar x_k \to 0$ as $k \to \infty$, and  $\bar x$ is a solution of the BVP \eqref{E}-\eqref{BVP}.
\end{proof}

\begin{proof}[\bf Proof of Corollary \ref{c:1}.] Assume that \eqref{L2} has a positive decreasing solution for $c=c_0>0$, and let $c_1\in (0, c_0)$. Then equation \eqref{L2} with $c=c_0$ is a Sturm majorant of \eqref{L2} with $c=c_1$, and from Theorem~\ref{Trec}, equation \eqref{L2} with $c=c_1$ has a positive decreasing solution. The application of Theorem \ref{Tmain} leads to the existence of a solution of \eqref{E}-\eqref{BVP} for $c=c_1$.
\end{proof}

Effective criteria for the solvability of BVP (\ref{E})-(\ref{BVP})  can be obtained by considering as a Sturm majorant of \eqref{L2} any linear equation that is known to have a global positive solution.

In the continuous case, a typical approach to obtaining global positivity of solutions
for equation
\begin{equation}\label{ER}
(t^{2}y')' + \gamma y=0,\quad t\geq 1,
\end{equation}
where $0<\gamma\leq 1/4$, is based on the Sturm theory. In virtue of the transformation $x=t^{2}y'$, this equation is equivalent  to the Euler equation
\begin{equation}\label{ERorig}
x''+\frac{\gamma}{t^{2}}x=0, \quad t\geq 1,
\end{equation}
whose general solutions are well-known.

In the discrete case, various types of Euler equations are considered  in the literature, see, e.g. \cite{HV,Naoto} and references therein.
It is somehow problematic to find a solution for some
natural forms of discrete Euler equations in the self-adjoint form \eqref{Lmin}.

Here our aim is to deal with solutions of Euler type equations.
\vskip4mm

\begin{lemma}\label{LeEuR}
The equation
\begin{equation}\label{EuR}
\Delta\bigl ((k+1)^{2}\Delta x_k \bigr) + \frac{1}{4}  x_{k+1}=0
\end{equation}
has a recessive solution which is positive decreasing on $\N$.
\end{lemma}
\begin{proof}
Consider the sequence
\[
y_k=\prod_{j=1}^{k-1}\frac{2j+1}{2j}, \quad k\ge1,
\]
with the usual convention $\prod_{j=1}^0u_j=1$. One can verify that
$y$ is a positive increasing solution of the equation
\begin{equation} \label{new-eu}
\Delta^2y_k+\frac{1}{2(k+1)(2k+1)}y_{k+1}=0
\end{equation}
on $\N$.

Set $x_k=\Delta y_k$. Then $x$ is a positive decreasing solution of the equation
\begin{equation}\label{Eunew2}
\Delta(2(k+1)(2k+1)\Delta x_k)+x_{k+1}=0
\end{equation}
on $\N$.
Obviously,
\[
2(k+1)(2k+1)\leq 4(k+1)^{2}, \quad k\geq 1,
\]
thus \eqref{Eunew2} is a Sturm majorant of \eqref{EuR}.
By Theorem \ref{Trec}, \eqref{EuR} has a recessive solution which is positive decreasing  on $\N$.
\end{proof}

Equation \eqref{EuR} can be understand as the reciprocal equation to the Euler difference equation
\begin{equation} \label{discr_euler}
\Delta^2 u_k+\frac{1}{4(k+1)^2}u_{k+1}=0,
\end{equation}
i.e., these equations are related by the substitution relation $u_k=d\Delta x_k$, $d\in\R$, where $u$ satisfies
\eqref{discr_euler} provided $x$ is a solution of \eqref{EuR}.
The form of \eqref{discr_euler} perfectly fits the discretization of the differential equation \eqref{ERorig} with $\gamma=1/4$, using the usual central difference scheme.

\begin{corollary}
\label{Cor1} Let (H$_{i}$), i=1,2,3, be satisfied and  $L_c$ be defined by \eqref{Lc}.
The BVP (\ref{E})-(\ref{BVP}) has at least one solution if  there exists
$\lambda>0$ such that  for $k\geq 1$
\begin{equation}\label{cc}
a_k\geq 4 \lambda (k+1)^{2},\quad \sqrt{1+c^{2}} \,L_c\, b_k \leq\lambda.
\end{equation}
\end{corollary}

\begin{proof}
Consider the equation \eqref{EuR}.
By Lemma \ref{LeEuR}, it has a positive decreasing solution on $\N$.
 The same trivially holds for the equivalent equation
\begin{equation}
\Delta\bigl (4 \lambda(k+1)^{2}\Delta x_k \bigr) +  \lambda x_{k+1}=0.
\label{Re}
\end{equation}
Since \eqref{cc} holds, \eqref{Re}  is a Sturm majorant of \eqref{L2}, and
by Theorem \ref{Trec}, equation \eqref{L2} has a positive decreasing solution on $\N$. Now the conclusion follows from Theorem \ref{Tmain}.
\end{proof}

\noindent{\textbf{Remark.}} Note that the sequence $b$ does not need to be bounded. For example,
consider as a Sturm majorant  of \eqref{L2} the equation
\begin{equation*}
\Delta\bigl(\lambda k 2^{k+1}\Delta x_k \bigr)+ \lambda 2^{k+1} x _{k+1}=0,\text{ \ \ }k\geq 0.
\end{equation*}
One can check that this equation has the solution $x_k=2^{-k}$.
This leads to the conditions
\[
a_k\geq \lambda k 2^{k+1}, \quad \ \sqrt{1+c^{2}} \,L_c\, b_k \leq\lambda 2^{k+1}\quad \text{for } k\geq 0
\]
ensuring the solvability of the BVP (\ref{E})-(\ref{BVP}).

Another criteria can be obtained by considering the equation
\begin{equation*}
\Delta\bigl( \lambda k^{3}\Delta x_k \bigr)+ \lambda \frac{k^{2}+3k+1}{k+2}x _{k+1}=0,\text{ \ \ }k\geq1
\label{Ex4}
\end{equation*}
having the solution $x_k=1/k$. This comparison with \eqref{L2} leads to the conditions
\[
a_k\geq \lambda k^{3},\quad \sqrt{1+c^{2}} \,L_c\, b_k \leq\lambda \frac{k^{2}+3k+1}{k+2}\quad \text{for }k\geq 1\, .
\]

The following example illustrates our result.

\begin{example}\label{Ex2} Consider the BVP
\begin{equation}
\begin{cases}
\Delta\bigl( (k+1)^{2}\Phi(\Delta x_k )\bigr)+  \dfrac{|\sin k|}{4\sqrt{2}\,k}\text{
}x^{3}_{k+1}=0,\text{ \ \ }k\geq1,\\[2mm]
x_1=c,\quad x_k>0,\quad \Delta x_k\le 0,
 \quad \displaystyle \lim_{k\to\infty}x_k=0.
\label{Ex1}
\end{cases}
\end{equation}

We have $L_c=c^2$, $a_k=(k+1)^2$, and
$b_k= \dfrac{|\sin k|}{4\sqrt{2}\,k}$.
Conditions in \eqref{cc} are fulfilled for any $c\in(0,1]$  when taking $\lambda=1/4$.
Indeed,
$$
a_k=(k+1)^2
=4\lambda(k+1)^2$$
and
$$
\sqrt{1+c^2}L_cb_k=\sqrt{1+c^2}c^2 b_k\le\sqrt{2}b_k
\le\frac{1}{4}|\sin k|\le \frac{1}{4}=\lambda.
$$
Corollary \ref{Cor1} now guarantees solvability of the BVP \eqref{Ex1} for any $c\in (0,1]$.
\end{example}

\section{Comments and open problems}\label{S5}

It is interesting to compare our discrete BVP with the continuous one investigated in  \cite{Trieste}.
Here  the BVP for the differential equation with the Euclidean mean curvature operator
\begin{equation}
\begin{cases}
\left(  a(t)\dfrac{x^{\prime}}{\sqrt{1+x^{\prime}{}^{2}}}\right)  ^{\prime
}+b(t)F(x)=0,\qquad t\in\lbrack1,\infty),\\
x(1)=1,\,x(t)>0,\,x^{\prime}(t)\leq0\text{ for }t\geq1,\,\displaystyle\lim
_{t\rightarrow\infty}x(t)=0,
\end{cases}
\tag{P}
\label{EC}
\end{equation}
has been considered. Sometimes solutions of differential equations satisfying the condition
\begin{equation*}
x(t)>0, \quad x'(t)\leq 0\, ,\quad  t\in[1,\infty),
\end{equation*}
are called \textit{Kneser solutions} and the problem to find such solution is called \textit{Kneser problem}.

The problem (\ref{EC}) has been studied under the following conditions:
\begin{itemize}
\item[(C$_{1}$)] The function $a$ is continuous on $[1,\infty)$, $a(t)>0$ in
$[1,\infty)$, and
\begin{equation*}
\int_{1}^{\infty}\frac{1}{a(t)}\,dt<\infty. \label{a}
\end{equation*}
\item[(C$_{2}$)] The function $b$ is continuous on $[1,\infty)$, $b(t)\geq0$
and
\begin{equation*}
\int_{1}^{\infty}b(t)\,\int_{t}^{\infty}\frac{1}{a(s)}dsdt<\infty. \label{B}
\end{equation*}
\item[(C$_{3}$)] The function $F$ is continuous on $\mathbb{R}$, $F(u)u>0$ for
$u\neq0$, and such that
\begin{equation}
\limsup_{u\rightarrow0^{+}}\frac{F(u)}{u}<\infty.\label{FF}
\end{equation}
\end{itemize}

The main result for solvability of \eqref{EC} is the following. Note that the principal solution for linear differential equation is defined similarly as the recessive solution, see e.g.  \cite{Trieste,H}.

\begin{theorem}
\label{Th second}{\rm\cite[Theorem 3.1]{Trieste}} Let (C$_{i}$), i=1,2,3, be verified and
\begin{equation*}
L=\underset{u\in(0,1]}{\sup}\frac{F(u)}{u}.\label{Fbar}
\end{equation*}
Assume
\begin{equation*}
\alpha=\inf_{t\geq1}\text{ }a(t)A(t)>1, \label{New}
\end{equation*}
where
\begin{equation*}
A(t)=\int_{t}^{\infty}\frac{1}{a(s)}\,ds. \label{A}
\end{equation*}
If the principal solution $z_{0}$ of the linear equation
\begin{equation*}
\left(  a(t)z^{\prime}\right)  ^{\prime}+\frac{\alpha}{\sqrt{\alpha^{2}-1}
}L \,b(t)z=0,\quad t\geq1, \label{EM1}
\end{equation*}
is positive and nonincreasing on $[1,\infty)$, then the BVP (\ref{EC}) has at
least one solution.
\end{theorem}
It is worth to note that the method used in \cite{Trieste} does not allow that $\alpha=1$ and thus Theorem \ref{Th second} is not immediately applicable when $a(t)=t^{2}$.
In \cite{Trieste} there are given several effective criteria for the solvability of the BVP \eqref{EC} which are similar to Corollary \ref{Cor1}. An example, which can be viewed
as a  discrete counterpart, is the above Example \ref{Ex2}.

\medskip
\noindent\textbf{Open problems.}
\vskip2mm
\noindent{\textbf{(1)}}
The comparison between Theorem \ref{Tmain} for the discrete BVP and
Theorem \ref{Th second} for the continuous one, suggests to investigate the BVP \eqref{E}-\eqref{BVP} on times scales.
\vskip2mm

\noindent{\textbf{(2)}} In \cite{Trieste}, the solvability of the continuous BVP has been proved under the weaker assumption \eqref{FF} posed on $F$. This is due to the fact that the set $\Omega$  is defined using a precise lower bound which is different from zero.
 It is an open problem if a similar estimation from below can be used in the discrete case and assumption \eqref{F} can be replaced  by \eqref{FF}.
\vskip2mm

\noindent{\textbf{(3)}} Similar BVPs concerning the existence of Kneser solutions
for difference equations with \textit{p}-Laplacian operator are considered
in \cite{JDEA2016} when $b_{k}<0$ for $k\in \Z^{+}.$ It should be
interesting to extend the solvability of the BVP \eqref{E}-\eqref{BVP} to
the case in which the sequence $b$ is negative and in the more general
situation when the sequence $b$ is of indefinite sign.

\section*{Acknowledgements}

The authors thank to anonymous referee for his/her valuable comments.
\vskip2mm

The research of the first and third authors has been supported by the grant GA20-11846S of the Czech Science Foundation. The second author was partially supported by Gnampa, National Institute for Advanced Mathematics (INdAM).

\section*{Acknowledgements}
The research of the first and third author has been supported by the grant GA20-11846S of the Czech Science Foundation.  The second author is  partially supported by Gnampa, National Institute for Advanced Mathematics (INdAM).

\end{document}